\documentclass[11pt]{amsart}
\usepackage[a4paper,left=2.5cm,right=2.5cm,top=2.5cm,bottom=2.5cm]{geometry}
\usepackage{amsfonts,amssymb,amsmath,latexsym,amsthm}
\usepackage{tikz-cd}
\usepackage{color}
\usepackage{hyperref}
\usepackage{url}
\newtheorem{definition}{Definition}[section]
\newtheorem{theorem}{Theorem}[section]

\newtheorem{lemma}{Lemma}[section]
\numberwithin{equation}{section}

\title{Toeplitz Operators on Fock Space $F_{\alpha}^{\infty}$}
\author{Sui Huang}
\address{School of Mathematical Sciences\\Chongqing Normal University\\Chongqing 401331\\ China}
\email{huangsui2@163.com}
\date{\today}

\begin{document}
	
	\begin{abstract}
		In this paper, we study  necessary and sufficient conditions for a positive
  Borel measure $\mu$ on the complex space  $\mathbb{C}$ to
  be a $(\infty,q)$ or $(p,\infty)$ (vanishing) Fock-Carleson measure
  through its Berezin transform.
  Then we discuss boundedness and compactness of the Toeplitz operator $T_{\mu}$
  with a positive Borel measure  $\mu$ as symbol on Fock space $F_{\alpha}^{\infty}$.
  Furthermore, we charaterize these properties of the Toeplitz operator  $T_{\varphi}$ with a symbol
  $\varphi$ which is in $BMO$.
	\end{abstract}
	\maketitle
\section{Introduction}
For any $\alpha>0$, let $d\lambda_\alpha(z)=\frac{\alpha}{\pi}e^{-\alpha|z|^2 }\text{d} A(z)$ be the Gaussian measure on the complex plane $\mathbb{C}$,
where $dA=dxdy$ is the area measure on $\mathbb{C}$.
When $0<p<\infty$, we use $L_{\alpha}^{p}$ to denote the function space consisting of Lebesgue measurable functions $f$ that satisfy
$\frac{p\alpha}{2\pi}\int_{\mathbb{C}}|f(z)e^{-\frac{\alpha|z|^2}{2}}|^p\text{d}A(w)<+\infty$ .
Define the norm of $f$ as
$$||f||_{p,\alpha}=[\frac{p\alpha}{2\pi}\int_{\mathbb{C}}|f(z)e^{-\frac{\alpha|z|^2}{2}}|^p\text{d}A(z)]^{\frac{1}{p}}.$$
When $p=+\infty$,$$L_{\alpha}^{\infty}=\{f|f(z)e^{-\frac{\alpha|z|^2}{2}}\in L^{\infty}(\mathbb{C},\text{d}A)\}.$$
If $f\in L_{\alpha}^{\infty}$, it's norm is defined as
$||f||_{\infty,\alpha}$=esssup$\{|f(z)e^{-\frac{\alpha|z|^2}{2}}|:~z\in \mathbb{C}\}$.
 Then $(L_{\alpha}^{p},||\cdot||_{p,\alpha})$ is a Banach space for $1\le p \le\infty$,
please see [1] for details.
\par
The Fock space $F_{\alpha}^{p}$ is a closed subspace of $L_{\alpha}^{p}$ consisting of all entire functions on $\mathbb{C}$,
 and $F_{\alpha}^{p}$ is a Banach space with the norm  $||\cdot||_{p,\alpha}$.
 In particular, $F_{\alpha}^{2}$ is a Hilbert space with
the inner product
$$<f,g>=\int_{\mathbb{C}}f(z)\overline{g(z)}\text{d}\lambda_{\alpha}(z)$$
for any $f,g\in F_{\alpha}^{2}$.
 The reproducing kernel of $F_{\alpha}^{2}$ is $K_{z}(\omega)=e^{\alpha\bar{z}\omega}$,
 and the orthogonal projection $P$ from $L_{\alpha}^{2}$ to $F_{\alpha}^{2}$ is defined
as
$$Pf(z)=<f,K_{z}>=\int_{\mathbb{C}}f(\omega)\overline{K_{z}(\omega)}\text{d}\lambda_{\alpha}(\omega).$$
For any $f\in L_{\alpha}^{2}$,\
$$Pf(z)=f(z)=<f,K_{z}>.$$
Let $k_z(\omega)=\frac{K_{z}(\omega)}{\sqrt[]{K_{z}(z)}}$ be the normalized kernel,
it is well known that $||k_z(\omega)||_{p,\alpha}=1$ when $1\leq p\leq +\infty$.
\par
Let $\varphi$ be a complex function  satisfying $\varphi K_{a}\in L_{\alpha}^{2}$.
We can define the Toeplitz operator with the symbol $\varphi$
\begin{align}\nonumber
  T_{\varphi}f(z)
  &=P(\varphi f)(z)\\
  &=\int_{\mathbb{C}}\varphi(\omega)f(\omega)\overline{K_{z}(\omega)}\text{d}\lambda_{\alpha}(\omega).\nonumber
\end{align}
Let $\mu$ be a positive Borel measure on $\mathbb{C}$. If
$$\int_{\mathbb{C}}|K_{z}(\omega)|^2e^{-\alpha|\omega|^2}\text{d}\mu(\omega)<+\infty$$
for any $z\in \mathbb{C}$, we say that $\mu$ satisfies condition (M).
\ Since $|K_{z}(\omega)|^2=|K_{2z}(\omega)|$, the above condition equates to
$$\int_{\mathbb{C}}|K_{z}(\omega)|e^{-\alpha|\omega|^2}\text{d}\mu(\omega)<+\infty.$$
We define the Toeplitz operator $T_{\mu}$ on $F_{\alpha}^2$ with the symbol  $\mu$ as
$$T_{\mu}f(z)=\frac{\alpha}{\pi}\int_{\mathbb{C}}f(\omega)\overline{K_{z}(\omega)}e^{-\alpha|\omega|^2}\text{d}\mu(\omega).$$
Extending the  definition of the Toeplitz operator to  $F_{\alpha}^{p}$  as following:
$$T_{\mu}f(z)=\frac{\alpha}{\pi}\int_{\mathbb{C}}f(\omega)\overline{K_{z}(\omega)}e^{-\alpha|\omega|^2}\text{d}\mu(\omega).$$
If $\text{d}\mu=\varphi\text{d}A$, then
\begin{align}\nonumber
  T_{\varphi}f(z)
  &=\frac{\alpha}{\pi}\int_{\mathbb{C}}f(\omega)\varphi(\omega)\overline{K_{z}(\omega)}\text{d}\lambda_{\alpha}(\omega)\\\nonumber
  &=\frac{\alpha}{\pi}\int_{\mathbb{C}}f(\omega)\overline{K_{z}(\omega)}e^{-\alpha|\omega|^2}\varphi(\omega)\text{d}A(\omega)\\\nonumber
  &=T_{\mu}f(z).
  \nonumber
\end{align}\nonumber
Compared with definition of the Toeplitz operator on $F_{\alpha}^2$,
the definition of the Toeplitz operator on $F_{\alpha}^p$ is loosely.
 If we denote $K=\text{span}\{K_{z}:z\in \mathbb{C}\}$, then $K$ is dense in $F_{\alpha}^p(0<p<+\infty)$ which makes
 $T_{\varphi}$ and $T_{\mu}$  well-defined on $K$, so we can generalize $T_{\varphi}$ and $T_{\mu}$ to $F_{\alpha}^p$.

Let $\mu$ be a positive Borel measure on $\mathbb{C}$ which satisfies  condition (M), say
\begin{align}\nonumber
  \tilde{\mu}_{t}(z)
  &=\frac{\alpha}{\pi}\int_{\mathbb{C}}|k_{z}(\omega)e^{-\alpha|\omega|^2}|^t\text{d}A(\omega)\\\nonumber
  &=\frac{\alpha}{\pi}\int_{\mathbb{C}}e^{-\frac{\alpha t|z-\omega|^2}{2} }\text{d}\mu(\omega)
  \nonumber
\end{align}\nonumber
is the $t$-Berezin transform of $\mu$. When $t=2$, $\tilde{\mu}(z)$ is the Berezin transform of $\mu$. If
$\text{d}\mu=f\text{d}A$,
\begin{align}\nonumber
  \tilde{f}_{t}(z)
  &=\frac{\alpha}{\pi}\int_{\mathbb{C}}|K_{z}(\omega)e^{-\frac{\alpha|\omega|^2}{2}}|^tf(\omega)\text{d}A(\omega)\\
  &=\frac{\alpha}{\pi}\int_{\mathbb{C}}f(\omega)e^{-\frac{\alpha t|z-\omega|^2}{2}}\text{d}A(\omega)\\
  &=B_{\frac{\alpha t}{2}}f(z),
  \nonumber
\end{align}
where $B_{\alpha}f(z)=\tilde{f}(z)=\frac{\alpha}{\pi}\int_{\mathbb{C}}f(\omega) e^{-\alpha |z-\omega|^2}\text{d}A(\omega)$ is the Berezin transform of $f$.

Fock spaces are also called Segal-Bargmann spaces. The study of Toeplitz operators on the classical Fock spaces started in the 1980s
 and has been fruitful so far, see [1-5].
 The properties of Toeplitz operators on the generalized Fock spaces and Fock-Sobolev spaces have been the subjects of comprehensive research over many years, see [6-11].
 The Berezin transform of measures, functions, and operators is one of important tools
 in the studies of  properties of the Toeplitz operators.
Isralowitz and Zhu  introduced the Fock-Carleson measures on $\mathbb{C}$ in [12].
  With  the Berezin transform of  measures, they discussed the properties of the Fock-Carleson measure,
  and studied  boundedness, compactness and Schatten classes of the Toeplitz operator $T_{\mu }$ on $F_{\alpha}^2$.
  Hu and Lv defined $(p,q)$ Fock-Carleson measures in [13], characterizing  Berezin transform of those measures,
boundedness and compactness of $T_{\mu}$ from $F_{\alpha}^{p}$ to $F_{\alpha}^q$.
Furthermore, they extended their results to the generalized Fock spaces in [8].
Mengestle discussed  boundedness and compactness of~$T_{\mu}$~between~$F_{\alpha}^p$~and $F_{\alpha}^1$, $F_{\alpha}^{\infty}$ in [14] on the basis of the works of Hu and Lv.

Based on the studies of Hu and Lv,
our work in this paper is to characterize  necessary and sufficient conditions for a positive Borel measures $\mu$ to be a $(\infty,q)$, $(p,\infty)$ (vanishing) Fock-Carleson measure
 by its $t$-Berezin transform  in Sect.2 and Sect.3.
Furthermore, we study  boundedness and compactness of the Toeplitz operators $T_{\mu}$ with symbols $\mu$
 and $T_{\varphi}$ with symbols $\varphi$ in $BMO$ on $F_{\alpha}^{\infty}$ in Sect.4.

\par
During discussing of this paper, we will use the following definitions and theorems.\\
\textbf{Definition $1.1^{[1]}$}
\textit{Let $r>0$, $\{a_k\}_{k=1}^{+\infty}$ be a sequence in $\mathbb{C}$. If}
\par$(1)\ \mathbb{C}=\bigcup\limits_{k=1}^{+\infty}B(a_k,r)$;
\par$(2)\ \{B(a_k,r)\}_{k=1}^{+\infty}$\ \textit{are pairwise disjoint;}
\par$(3)$ \textit{For any  $\delta >0$,
$z$ belongs to at most $m$ sets in $\{B(a_k,r)\}_{k=1}^{+\infty}$, \\
then $\{a_k\}$ is said to be an $r$-lattice of $\mathbb{C}$.}\\
\textbf{Lemma $1.2$}\textsuperscript{[1]}\ \textit{Let $\alpha>0$, $p>0$.
If $f(z)$ is an entire function on $\mathbb{C}$, there exists a constant $C>0$ such that
$$|f(a)e^{-\frac{\alpha|a|^2}{2}}|^p\leqslant \frac{C}{r^2}\int_{B(a,r)}|f(z)e^{-\frac{\alpha|z|^2}{2}}|^p\text{d}A(z) $$
for all $r>0$ .}\\
\textbf{Lemma $1.3$}\textsuperscript{[8]}
\textit{Let $1\leqslant p < +\infty$ and $\mu$ be a positive Borel measure on $\mathbb{C}$. The following propositions are equivalent:}
\par$(1)\ \tilde{\mu}_t\in L^p( \mathbb{C},\text{d}A),\ t>0$;
\par$(2)\ \mu(B(z,\delta)) \in L^p( \mathbb{C},\text{d}A),\ \delta>0$;
\par$(3)\ \{ \mu(B(a_k,r))\}\in l^p,\ r>0$.\\
\textbf{Lemma $1.4$}\textsuperscript{[8]}\ \textit{Let $0<q<p<+\infty$ and $\mu$ be a positive Borel measure on $\mathbb{C}$,
$s=\frac{p}{q}$, $\frac{1}{s}+\frac{1}{s'}=1$.  The following propositions are equivalent:}
\par$(1)$ $ \mu $ is $(p,q)$\ \textit{ Fock-Carleson measure;}
\par$(2)$ $ \mu $ is $(p,q)$\  \textit{vanishing Fock-Carleson measure;}
\par$(3)$ $ \tilde{\mu}_t\in L^{s'}( \mathbb{C},\text{d}A)$, $t>0$;
\par$(4)$ $ \mu(B(z,\delta))\in L^{s'}( \mathbb{C},\text{d}A)$, $\delta>0$;
\par$(5)$ $ \mu(B(z,\delta)) \in l^{s'}$, $r>0$.
\par
 Let $BMO$ be the bounded average oscillation function space.\\
\textbf{Lemma $1.5$}\textsuperscript{[1]}
 \textit{Let $f\in BMO$,~$\alpha>0$,~$\beta>0$.~$\widetilde{f_{\alpha}} \in C_0( \mathbb{C})$ if and only if $\widetilde{f_{\beta}}\in C_0( \mathbb{C})$.}
\par In this paper, we use the symbol $C$ to denote positive constants.

\section{$(\infty,q)$ Fock-Carleson Measure }
We discuss necessary and sufficient  conditions for a positive Borel measure $\mu$ on the complex plane $\mathbb{C}$ to be a $(\infty,q)$ (vanishing) Fock-Carleson measure in this section.

\begin{definition} \textit{Let $\mu$ be a positive Borel measure on $\mathbb{C}$.
Define $$L_{\alpha}^p(\text{d}\mu)=\{f|\int_{\mathbb{C}}|f(z)e^{-
\frac{\alpha|z|^2}{2}}|^p\text{d}\mu<+\infty\}$$
is the function space consisting of Lebesgue measurable functions $ f $ on $ \mathbb{C}$ satisfying the above condition.
For $ 1 \leqslant p < +\infty $, the norm of $f \in L_{\alpha}^{p}(\text{d}\mu)$ can be defined as following
$$||f||_{p,\mu}=[\int_{\mathbb{C}}|f(z)e^{-\frac{\alpha|z|^2}{2}}|^p\text{d}\mu]^{\frac{1}{p}}.$$
Then $(L_{\alpha}^{p}(\text{d}\mu),||\cdot||_{p,\mu})$ is a Banach space.
 When $p=+\infty$,
$L_{\alpha}^{\infty}(\text{d}\mu)$ denotes the space consisting of the essentially bounded measurable functions on $\mathbb{C}$ about $\mu$.
If $f\in L_{\alpha}^{\infty}(\text{d}\mu)$, it's norm is defined as
$$||f||_{\infty,\mu}=\inf \limits_{E_0\subset C \atop \mu(E_0)=0}
(\sup\limits_{C-E_0}|f(z)e^{- \frac{\alpha|z|^2}{2}}|)=\text{esssup}_{z\in \mathbb{C}}|f(z)e^{-\frac{\alpha|z|^2}{2}}|.$$
Under the  norm, $L_{\alpha}^{\infty}(\text{d}\mu)$ is also a Banach space.}
\end{definition}

\begin{definition}\textit{Let $0<p,q<\infty$ and $\mu$ be a positive Borel measure on the complex plane $\mathbb{C}$.
If $i:F_{\alpha}^{p}\longrightarrow L_{\alpha}^{q}(\text{d}\mu)$ is bounded, that is, there exists a constant  $C>0$ for any $f\in F_{\alpha}^{p}$ such that
$$[\int_{\mathbb{C}}|f(z)e^{-\frac{\alpha|z|^2}{2}}|^q\text{d}\mu(z)]^{\frac{1}{q}}\leqslant C||f||_{p,\alpha},$$
 we call $\mu$ a $(p,q)$ Fock-Carleson measure.}

 \textit{When $p=+\infty,\ f\in F_{\alpha}^{\infty}$
$$[\int_{\mathbb{C}}|f(z)e^{-\frac{\alpha|z|^2}{2}}|^q\text{d}\mu(z)]^{\frac{1}{q}}\leqslant C||f||_{\infty,\alpha},$$
 we call $\mu$ a $(\infty,q)$ Fock-Carleson measure.}

 \textit{When $q=+\infty$, $f\in F_{\alpha}^p$
$$||f||_{\infty,\mu}\leqslant C||f||_{p,\alpha},$$
 we call $\mu$ a $(p,\infty)$ Fock-Carleson measure.}
\end{definition}

\begin{definition}\ \textit{ Let $0<q<\infty$ and $\mu$ be a positive Borel measure on the complex plane $\mathbb{C}$.
If
$$\lim_{n\to+\infty}\int_{\mathbb{C}}|f_{n}(z)e^{-\frac{\alpha|z|^2}{2}}|^q\text{d}\mu(z)=0,$$
we call $\mu$  a $(\infty,q)$ vanishing Fock-Carleson measure,
 where $\{f_{n}\}_{n=1}^{+\infty}$ is a bounded sequence in $F_{\alpha}^{\infty}$ and converges  to $0$ uniformly on any compact subset of ~$\mathbb{C}$.
If $\{f_n\}_{n=1}^{+\infty}$ is a bounded sequence  in $F_{\alpha}^p$ and converges  to $0$ uniformly on any compact subset of ~$\mathbb{C}$, such that
$$\lim\limits_{n\to+\infty}||f_n||_{\infty,\mu}=0,$$
 we call $\mu$ a $(p,\infty)$ vanishing Fock-Carleson measure.}
\end{definition}

\begin{lemma}\textit{Let $\mu$ be a positive Borel measure on $\mathbb{C}$. If $f\in \displaystyle{\bigcup_{p\geqslant1}}F_{\alpha}^{p}$, then
$$\lim_{p\to+\infty}||f||_{p,\alpha}=||f||_{\infty,\alpha}.$$}
\end{lemma}
\begin{proof}
 When~$1\leqslant p<q<+\infty$, we know that $F_{\alpha}^{p}\subset F_{\alpha}^{q} \subset F_{\alpha}^{\infty}$ by Theorem 2.7 in [1].
If there is a $p_0\geqslant 1$ such that $f\in F_{\alpha}^{p_0}$, then $f\in F_{\alpha}^p$~for any~$p>p_0$ and
$$|f(z)e^{-\frac{\alpha|z|^2}{2}}|\leqslant||f||_{p,\alpha},$$
therefore
$$||f||_{\infty,\alpha}\leqslant \varliminf_{p\to+\infty}||f||_{p,\alpha}.$$
\par
Let $p>2$,
\begin{align}\nonumber
 ||f||_{\infty,\alpha}^p
 &=\frac{p\alpha}{2\pi}\int_{\mathbb{C}}|f(z)e^{-\frac{\alpha|z|^2}{2}}|^p\text{d}A(z)\\
  &=\frac{p}{2}\int_{\mathbb{C}}|f(z)e^{-\frac{\alpha|z|^2}{2}}|^{p-2}|f(z)|^2\text{d}\lambda_{\alpha}(\omega).\\
  &\leqslant\frac{p}{2}||f||_{\infty,\alpha}^{p-2}||f||_{2,\alpha}^2,
  \nonumber
\end{align}
then
$$||f||_{p,\alpha}\leqslant[\frac{p}{2}||f||_{2,\alpha}^{2}]^{   \frac{1}{p}}||f||_{\infty,\alpha}^{\frac{p-2}{p}},$$
and
$$\varlimsup_{p\to+\infty}||f||_{p,\alpha}\leqslant||f||_{\infty,\alpha} .$$
Thus $\lim\limits_{p\to +\infty}||f||_{p,\alpha}=||f||_{\infty,\alpha}$.\ $\qedsymbol$
\end{proof}
\begin{theorem}
 \textit{Let $\mu$ be a positive Borel measure on $\mathbb{C}$.
 $\mu$ is a $(\infty,q)$ Fock-carleson measure if and only if $\mu$ is a $(p,q)$\ Fock-Carleson measure $(p\geqslant1)$.}
\end{theorem}
\begin{proof}Since $F_{\alpha}^p\subset F_{\alpha}^{\infty}$, $||f||_{\infty,\alpha}\leqslant||f||_{p,\alpha}$ for any $f\in F_{\alpha}^p$.
If $\mu$ is a $(\infty,q)$ Fock-Carleson measure, then there exists a constant $C>0$ for $f\in F_{\alpha}^{p}$ such that
$$(\int_{\mathbb{C}}|f(z)e^{-\frac{\alpha|z|^2}{2}}|^{q}\text{d}\mu)^{\frac{1}{q}}\leqslant C||f||_{\infty,\alpha} \leqslant C||f||_{p,\alpha}.$$
So $\mu$ is a $(p,q)$\ Fock-Carleson measure.

Conversely, if $\mu$ is a $(p,q)$\ Fock-Carleson measure for any $p\geqslant1$, then
$$(\int_{\mathbb{C}}|f(z)e^{-\frac{\alpha|z|^2}{2}}|^{q}\text{d}\mu(z))^{\frac{1}{q}}\leqslant C||f||_{p,\alpha}$$
for all $f\in F_{\alpha}^{p}$. From Lemma 2.4,
$$\lim_{p\to+\infty}||f||_{p,\alpha}=||f||_{\infty,\alpha} ,$$
therefore
$$[\int_{\mathbb{C}}|f(z)e^{-\frac{\alpha|z|^2}{2}}|^{q}\text{d}\mu(z)]^{\frac{1}{q}}\leqslant C||f||_{\infty,\alpha}.$$
So $\mu$ is a $(\infty,q)$ Fock-Carleson measure.$\qedsymbol$
\end{proof}

\begin{theorem}Let $\mu$ be a positive Borel measure on $\mathbb{C}$, and $\{a_{k}\}_{k=1}^{\infty}$ be an $r$-lattice of~$\mathbb{C}$.
The following propositions are equivalent:
\par$(1)$ $\mu$ is $(\infty,q)$ Fock-Carleson measure;
\par$(2)$ $\tilde{\mu_{t}}\in L^1(\mathbb{C},dA)$, $t>0$;
\par$(3)$ $\mu(B(z,r))\in L^1(\mathbb{C},dA)$;
\par$(4)$ $\{\mu(B(a_{k},r))\}\in l^1$.
\end{theorem}
\begin{proof} $(1)\Rightarrow (2)$.\
Let $\mu$ be a $(\infty,q)$ Fock-Carleson measure. Since $K_z\in F_{\alpha}^{\infty}$,
so $$\int_{\mathbb{C}}|K_{z}(\omega)e^{-\frac{\alpha|\omega|^2}{2}}|^q\text{d}\mu(\omega)\leqslant C||K_{z}||_{\infty,\alpha}^q.$$
Then
\begin{align}\nonumber
\widetilde{\mu}_{q}(z)
&=\frac{\alpha}{\pi}\int_{\mathbb{C}}|K_{z}(\omega)e^{-\frac{\alpha|\omega|^2}{2}}|^q\text{d}\mu(\omega)\\
&=\frac{\alpha}{\pi} e^{-\frac{\alpha q|z|^2}{2}} \int_{\mathbb{C}}|K_{z}(\omega)  e^{-\frac{\alpha|\omega|^2}{2}}|^q\text{d}\mu(\omega)\\
&\leqslant C \frac{\alpha}{\pi} e^{-\frac{\alpha q|z|^2}{2}} ||K_{z}||_{\infty,\alpha}^{q}.
\nonumber
\nonumber
\end{align}
According to Lemma 1.3,~$\widetilde{\mu}_{q}(z)\in L^1(\mathbb{C},\text{d}A)$,
so $\widetilde{\mu}_{t}(z)\in L^1(\mathbb{C},\text{d}A)$.\par
\indent $(2)\Rightarrow(3).$\ Through the direct calculation, we can get
\begin{align}\nonumber
 \mu(B(z,r))
 &=\int_{B(z,r)}1\text{d}\mu(\omega)\\
  &=\int_{B(z,r)}e^{-\frac{\alpha t|z-\omega|^2}{2}}e^{\frac{\alpha t|z-\omega|^2}{2}}\text{d}\mu(\omega)\\
  &\leqslant e^{\frac{\alpha t r^2}{2}}\int_{\mathbb{C}}e^{-\frac{\alpha t|z-\omega|^2}{2}}\text{d}\mu(\omega)\\
  &=e^{\frac{\alpha tr^2}{2}}\widetilde{\mu}_{t}(z).
  \nonumber
\end{align}
$\text{If }\tilde{\mu}_t(z)\in L^1(\mathbb{C},\text{d}A),\ \mu(B(z,r))\in L^1(dA)$.

\indent$(3)\Rightarrow(4)$.\ According to the estimate of Theorem 3.3 in [8], there is a constant $C>0$ such that
$$\mu(B(a_{k},r))\leqslant C\int_{B(a_{k},r)}\mu(B(\omega,\delta))\text{d}A(\omega),$$
 so
\begin{align}\nonumber
\sum_{k=1}^{+\infty}\mu(B(a_{k},r))
&\leqslant C\sum_{k=1}^{+\infty}\int_{B(a_{k},r)}\mu(B(\omega,\delta))\text{d}A(\omega)\\
  &\leqslant CN\int_{\mathbb{C}}\mu(B(\omega,\delta))\text{d}A(\omega)\\
  &< +\infty,
  \nonumber
\end{align}
therefore $\{\mu(B(a_{k},r)\}\in l^1$.
Because $\mathbb{C}=\cup_{k=1}^{+\infty}B(a_{k},r)$,
so $\mu(\mathbb{C})\leqslant\sum_{k=1}^{+\infty}\mu(B(a_{k},r)<C$. At this point $\mu$ is a finite Borel measure on $\mathbb{C}$.\\
\indent $(4)\Rightarrow (1)$.\
Let $\{\mu(B(a_k,r))\}\in l^1$. When $f\in F_{\alpha}^{\infty}$,
\begin{align}\nonumber
\int_{\mathbb{C}}|f(z)e^{-\frac{\alpha|z|^2}{2}}|^q\text{d}\mu(z)
&\leqslant\sum_{k=1}^{+\infty}\int_{B(a_k,r)}|f(z)e^{-\frac{\alpha|z|^2}{2}}|^q\text{d}\mu(z)\\
&\leqslant||f||_{\infty,\alpha}^{q}\sum_{k=1}^{+\infty}\mu(B(a_k,r))\\
&\leqslant C||f||_{\infty,\alpha}^q.
  \nonumber
\end{align}
It shows that $\mu$ is a $(\infty,q)$ Fock-Carleson measure . \ $\qedsymbol$
\end{proof}
\begin{theorem}
\textit{Let $\mu$ be a positive Borel measure on $\mathbb{C}$.
If $\mu$ is a $(\infty,q)$ Fock-Carleson measure, then $\widetilde{\mu_{t}}(z)\in L^{\infty}(dA)$ for $t>0.$}
\end{theorem}
\begin{proof}
Let $\mu$ be a $(\infty,q)$ Fock-Carleson measure, then
for any $f\in F_{\alpha}^{\infty}$, there exists $C>0$
such that
$$\int_{\mathbb{C}}|f(z)e^{-\frac{\alpha|z|^2}{2}}|^q\text{d}\mu(\omega)\leqslant C||f||_{\infty,\alpha}^{q}.$$
Taking $f=k_z$, then
\begin{align}\nonumber
\widetilde{\mu}_q(z)
&=\frac{\alpha}{\pi}\int_{\mathbb{C}}|k_{z}(\omega)e^{-\frac{\alpha|\omega|^2}{2}}|^q\text{d}\mu(\omega)\\
&\leqslant \frac{\alpha}{\pi} C||k_z||_{\infty,\alpha}^q\\
&=\frac{\alpha C}{\pi}.
\nonumber
\end{align}
It can be known that $\widetilde{\mu_{t}}(z)\in L^{\infty}(\mathbb{C},\text{d}A)$. \ $\qedsymbol$
\end{proof}
\begin{theorem}
\textit{Let~$\mu$~be a positive Borel measure on~$\mathbb{C}$.
$\mu$ is a $(\infty,q)$ vanishing Fock-Carleson measure if and only if $\mu$ is a $(p,q)$ vanishing Fock-Carleson measure.}
\end{theorem}
\begin{proof}
Let $\{f_{n}\}$ be a  bounded sequence in $F_{\alpha}^p$ and converge to 0 uniformly on any compact subset of $\mathbb{C}$,
then $\{f_n\}$ is  bounded  in $F_{\alpha}^{\infty}$. Since $\mu$ is a $(\infty,q)$ vanishing Fock-Carleson measure, we have that
$$\lim_{n\to +\infty}\int_{\mathbb{C}}|f_n(z)e^{-\frac{\alpha|z|^2}{2}}|^q\text{d}\mu(z)=0,$$
so $\mu$\ is $(p,q)$ vanishing Fock-Carleson measure.
\par
Conversely, if $\mu$ is a $(p,q)$ vanishing Fock-Carleson measure for any $p\geqslant1$ , and
$\{f_n\}$ is bounded  sequence in $F_{\alpha}^p$ which converges  to 0 uniformly on any subset of $\mathbb{C}$, then
$$\lim_{n\to +\infty}\int_{\mathbb{C}}|f_{n}(z)e^{-\frac{\alpha|z|^2}{2}}|^q\text{d}\mu(z)=0.$$
Because $\lim\limits_{n\to +\infty}||f_n||_{p,\alpha}=||f_n||_{\infty,\alpha}$,
 and if$\{f_n\}$ is a bounded function sequence in $F_{\alpha}^p$, then $\{f_n\}$ is bounded in $F_{\alpha}^{\infty}$.
Thus $$\lim_n\int_{\mathbb{C}}|f_{n}(z)e^{-\frac{\alpha|z|^2}{2}}|^q\text{d}\mu(z)=0,$$
we can get that $\mu$ is a $(\infty,q)$ vanishing  Fock-Carleson measure.\ $\qedsymbol$
\end{proof}
\section{$(p,\infty)$ Fock-Carleson Measure}
In this section we will discuss necessary and sufficient  conditions for a positive Borel measure $\mu$ on  $\mathbb{C}$ to be
 a $(p,\infty)$ (vanishing) Fock-Carleson measure.
 $\mu$ needs to satisfy the little stronger condition that $\mu(\mathbb{C})$ is finite.
  We prove the following theorem firstly.


\begin{lemma}
\textit{Let $\mu$ be a positive finite Borel measure, then~$\lim\limits_{p\to+\infty}||f||_{p,\mu}=||f||_{\infty,\mu}$
~for any $f\in \displaystyle{\bigcup_{p\geqslant1}}L_{\alpha}^p(\text{d}\mu)$.}
\end{lemma}
\begin{proof}
Let $f\in L_{\alpha}^{\infty}(\text{d}\mu)$,
then $|f(z)e^{-\frac{\alpha|z|^2}{2}}|\leqslant ||f|||_{\infty,\mu}$.
It can be obtained
\begin{align}\nonumber
||f||_{p,\mu}^p
&=\int_{\mathbb{C}}|f(z)e^{-\frac{\alpha|z|^2}{2}}|^p\text{d}\mu(z)\\
&\leqslant \int_{\mathbb{C}}||f||_{\infty,\mu}^p\text{d}\mu(z) \\
&\leqslant ||f||_{\infty,\mu}^p\mu(\mathbb{C}),
\end{align}
thus $$\varliminf_{p\to +\infty}||f||_{p,\mu}\leqslant||f||_{\infty,\mu}.$$

On the other hand, we denote $E_{\varepsilon}=\{z\in \mathbb{C} |  |f(z)|\geqslant ||f||_{\infty,\mu}-\varepsilon\}$ for any $\varepsilon>0$,
then $\mu(E_{\varepsilon})>0$. If $\mu(E_{\varepsilon})=0$,
then $|f(z)|\leqslant ||f||_{\infty,\mu}-\varepsilon$ for any $z\in E-E_{\varepsilon}$.
This contradicts the definition of $||f||_{\infty,\mu}$. At this moment
\begin{align}\nonumber
||f||_{p,\mu}^p
&=\int_{\mathbb{C}}|f(z)e^{-\frac{\alpha|z|^2}{2}}|^p\text{d}\mu\\
&\geqslant \int_{E_{\varepsilon}}|f(z)e^{-\frac{\alpha|z|^2}{2}}|^p\text{d}\mu \\
&\geqslant (||f||_{\infty}-\varepsilon)^p\mu(E_{\varepsilon}),
\end{align}
so $$(||f||_{\infty}-\varepsilon)(\mu(E_{\varepsilon}))^{\frac{1}{p}}\leqslant ||f||_{p,\mu},$$
then
$$||f||_{\infty}\leqslant \varliminf_{p\to +\infty}||f||_{p,\mu}.$$
This means $||f||_{\infty,\mu}=\lim\limits_{p\to+\infty}||f||_{p,\mu}.$
\ $\qedsymbol$
\end{proof}
\begin{theorem}
\textit{Let $\mu$ be a positive finite Borel measure on $\mathbb{C}$.
 $\mu$ is a $(p,\infty)$ Fock-Carleson measure if and only if $\mu$ is a $(p,q)$ Fock-Carleson measure.}
\end{theorem}
\begin{proof}
 Let $\mu$ be a $(p,\infty)$ Fock-Carleson measure firstly. For any $f\in F_{\alpha}^{p}$, there is a constant $C>0$ such that
$$||f||_{\infty,\mu}=\text{esssup}
_{z\in \mathbb{C}}|f(z) e^{-\frac{\alpha |z|^2}{2}}|\leqslant C||f||_{p,\alpha}.$$
Because
$$\lim\limits_{q\to+\infty}||f||_{q,\mu}=||f||_{\infty,\mu},$$
we can get $$\int_{\mathbb{C}}|f(z)e^{-\frac{\alpha|z|^2}{2}}|^q\text{d}\mu \leqslant ||f||_{p,\alpha}^q.$$
So $\mu$ is $(p,q)$ Fock-Carleson measure.\par
We now prove the necessity. Let $\mu$ be a $(p,q)$ Fock-Carleson measure, then
 there exists constant $C>0$ for any $f\in F_{\alpha}^p$ such that
 $$||f||_{q,\mu}\leqslant C||f||_{p,\alpha}.$$
Because $$\lim\limits_{q\to +\infty}||f||_{q,\mu}=||f||_{\infty,\mu},$$
so $$||f||_{\infty,\mu}\leqslant C||f||_{p,\alpha},$$
 we can obtain that $\mu$ is a $(p,\infty)$ Fock-Carleson measure.\ $\qedsymbol$
\end{proof}
\begin{theorem}
\textit{Let $\mu$ be a positive finite measure Borel measure on $\mathbb{C}$.
$\mu$ is a $(p,\infty)$ vanishing Fock-Carleson measure
if and only if $\mu$ is a $(p,q)$ vanishing Fock-Carleson measure for any $q\geqslant p\geqslant1$ .}
\end{theorem}
\begin{proof}
The discussions are similar in Theorem 2.8, we omitted here.
\end{proof}
\section{Toeplitz Operators on $F_{\alpha}^{\infty}$}
In this section, we will focus on  boundedness and compactness of the Toeplitz operator $T_{\mu}$ on $F_{\alpha}^{\infty}$ with the symbols $\mu$ which satisfying the condition (M).
Let $\mu$ be a positive finite Borel measure on $\mathbb{C}$, and define Toeplitz operator $T_{\mu}$ with $\mu $ on $F_{\alpha}^{\infty}$ as following
$$T_{\mu}f(z)=\int_{\mathbb{C}} f(\omega) \overline{K_{z}(\omega)}e^{-\alpha|\omega|^2}\text{d}\mu(\omega),$$
where $f\in F_{\alpha}^{\infty}$.

\begin{theorem}
\textit{Let $\mu$ be a positive Borel measure on $\mathbb{C}$.
$T_{\mu}$ is bounded on $F_{\alpha}^{\infty}$ if and only if $\tilde{\mu}_1(t)\in L^{\infty}(\mathbb{C},\text{d}A)$.}
\end{theorem}
\begin{proof} For any $f\in F_{\alpha}^{\infty}$,
\begin{align}\nonumber
|T_{\mu}f(z)e^{-\frac{\alpha|z|^2}{2}}|
&=|\int_{\mathbb{C}}f(\omega)\overline{K_{z}(\omega)}e^{-\alpha|\omega|^2}e^{-\frac{\alpha|z|^2}{2}}\text{d}\mu(\omega)|\\
&\leqslant \int_{\mathbb{C}}|f(\omega)e^{-\frac{\alpha|\omega|^2}{2}}| |k_z(\omega)e^{-\frac{\alpha|\omega|^2}{2}}|\text{d}\mu(\omega) \\
&\leqslant ||f||_{\infty,\alpha}\int_{\mathbb{C}}|k_z(\omega)e^{-\frac{\alpha|\omega|^2}{2}}|\text{d}\mu(\omega) \\
&=||f||_{\infty,\alpha}\tilde{\mu}_1(z).
\end{align}
If $\tilde{\mu}_1(z)\in L^{\infty}(\mathbb{C},\text{d}A) $, then $T_{\mu}f\in F_{\alpha}^{\infty}$.
\par
On the contrary,                                                                                                                                                                                                                  let $T_{\mu}$ be bounded on $F_{\alpha}^{\infty}$, thus
\begin{align}\nonumber
\tilde{\mu}_2(z)
&=\int_{\mathbb{C}}|k_z(\omega)e^{-\frac{\alpha|\omega|^2}{2}}|^2\text{d}\mu(\omega)\\
&=|<T_{\mu}k_z,k_z>| \\
&\leqslant ||T_{\mu}k_z||_{\infty,\alpha}||k_z||_{1,\alpha}\\
&\leqslant||T_{\mu}||,
\end{align}
then $\tilde{\mu}_2(z)\in L^{\infty}(\mathbb{C},\text{d}A)$.
By Lemma 1.3, we have $\tilde{\mu}_1(z)\in L^{\infty}(\mathbb{C}, \text{d}A)$.\ $\qedsymbol$
\end{proof}
\begin{theorem}\ \textit{Let $\mu$ be a positive Borel measure on $\mathbb{C}$.
$T_{\mu}$ is a compact operator on $F_{\alpha}^{\infty}$ if and only if $\tilde{\mu}_1(z)\in C_0(\mathbb{C})$.}
\end{theorem}
\begin{proof} Let $\{f_n\}$ be a bounded sequence in $F_{\alpha}^{\infty}$ and converge  to 0 uniformly on any compact subsets of $\mathbb{C}$.
 By Theorem 4.1
$$
|T_{\mu}f_n(z)e^{-\frac{\alpha|z|^2}{2}}|
\leq||f_n||_{\infty,\alpha}\tilde{\mu}_1(z),\\
$$
then $||T_{\mu}f_n||_{\infty,\alpha}\leq ||f_n||_{\infty,\alpha}\tilde{\mu}_1(z).$
If $\tilde{\mu}_1(z)\in C_0(\mathbb{C})$,
then $\lim\limits_{n\rightarrow +\infty}||T_{\mu}f_n||_{\infty,\alpha}=0$,
thus $T_{\mu}$ is a compact operator on $F_{\alpha}^{\infty}$.
\par
On the contrary, let $f_n=k_z$, thus $||k_z||_{\infty,\alpha}=1$.\
When $|\omega|\leqslant R$,  $\{k_z(\omega)\}$ converges to 0 uniformly as $z$ converges to $\infty$.
$T_{\mu}$ is compact, then
$$\lim\limits_{z\to \infty}||T_{\mu}k_z||_{\infty,\alpha}=0.$$
And  \begin{align}\nonumber
\tilde{\mu}_{2}(z)=|<T_{\mu}k_z,k_z>|
&\leqslant ||T_{\mu} k_z||_{\infty,\alpha}||k_z||_{1,\alpha}\\
&\leqslant ||T_{\mu} k_z||_{\infty,\alpha}.
\end{align}
Since $\lim\limits_{z\to \infty} ||T_{\mu} k_z||_{\infty,\alpha}
=0 $,\
so $\tilde{\mu}_2(z)\in C_0(\mathbb{C})$.\
By Lemma 1.3, $\tilde{\mu}_1(z)\in C_0(\mathbb{C})$.\qedsymbol
\end{proof}
\par
Let $\varphi \geqslant 0$. $\varphi$  satisfys the condition $(\uppercase\expandafter{\romannumeral1}_1)$ in $[1]$ if
$$\int_{\mathbb{C}}|\varphi(\omega)||K_z(\omega)|^2\text{d}\lambda_\alpha(\omega)<+\infty.$$
If $\text{d}\mu=\varphi\text{d}A $, we define    $T_{\varphi}$ on $F_{\alpha}^{\infty}$ as
$$T_{\varphi}f(z)=\int_{\mathbb{C}}\varphi(\omega)f(\omega)\overline{K_z(\omega)}\text{d}\lambda_{\alpha}(\omega),$$
then $T_{\varphi}=T_{\mu}$, and $\tilde{\varphi}_t(z)=\tilde{\mu}_t(z)$.\
\begin{theorem}
\textit{Let $\varphi \geqslant 0$ and $\varphi$ satisfys the conditions$(\uppercase\expandafter{\romannumeral1}_1)$.
  $T_{\varphi}$ is a bounded operator on $F_{\alpha}^{\infty}$ if and only if $\tilde{\varphi}_t\in L^{\infty}(\mathbb{C},\text{d}A)$ for $t>0$.}
\end{theorem}
\begin{proof}
 For any $f\in F_{\alpha}^{\infty}$,\
\begin{align}\nonumber
| T_{\varphi}f(z)e^{-\frac{\alpha|z|^2}{2}}|
&= |  \int_{\mathbb{C}} \varphi(\omega) f(\omega)\overline{K_z(\omega)} \text{d}\lambda_\alpha(\omega) e^{-\frac{\alpha|z|^2}{2}}|\\
    &\leqslant  \frac{\alpha}{\pi}  \int_{\mathbb{C} }\varphi(\omega) |f(\omega)| |e^{\alpha \overline{z}  \omega}|e^{-\alpha|\omega|^2-\frac{\alpha|z|^2}{2}}\text{d}A(\omega) \\
&=\frac{\alpha}{\pi}  \int_{\mathbb{C}} \varphi(\omega) |f(\omega) e^{-\frac{\alpha|\omega|^2}{2}} |  e^{-\frac{\alpha|z-\omega|^2}{2} }\text{d}A(\omega)\\
&\leqslant ||f||_{\infty,\alpha}\frac{\alpha}{\pi} \int_{\mathbb{C} }\varphi(\omega)  e^{-\frac{\alpha|z-\omega|^2}{2} }\text{d}A(\omega)\\
&=2||f||_{\infty,\alpha} \tilde{\varphi}_{1}(z).
\end{align}
Since $\varphi \geqslant 0$, $\tilde{\varphi}_t\in L^{\infty}(\mathbb{C},\text{d}A)$ equivalent to
$\tilde{ \varphi}_{1}\in L^{\infty}(\mathbb{C},\text{d}A)$, so $||T_{\varphi} f||_{\infty,\alpha}\leqslant 2||f||_{\infty,\alpha} ||\tilde{\varphi}_{1}||_{\infty}$,
then  $||T_{\varphi}||\leqslant 2||\tilde{\varphi}_{1}||_{\infty}$.
\par
On the contrary,
\begin{align}\nonumber
 \tilde{\varphi}_{2}(z)
&= <T_{\varphi} k_z,k_z >\\
    &\leqslant  ||T_{\varphi} k_z||_{\infty,\alpha} ||k_z||_{1,\alpha}\\
&\leqslant||T_{\varphi}||.
\end{align}
So $\tilde{\varphi}_2 $ is bounded, then $\tilde{\varphi}_t\in L^{\infty}(\mathbb{C},\text{d}A)$.\ \qedsymbol
\end{proof}
\begin{theorem}
 \textit{Suppose $\varphi \geqslant 0$, $\varphi$ satisfys the conditions $(\uppercase\expandafter{\romannumeral1_1})$.
$T_{\varphi}$ is a compact operator on $F_{\alpha}^{\infty}$ if and only if $\tilde{\varphi}_{t}\in C_0(\mathbb{C})$ for $t>0. $}
\end{theorem}
\begin{proof}
\begin{align}\nonumber
 0\leqslant \tilde{\varphi}_{2}(z)
&\leqslant <T_{\varphi} k_z,k_z >\\
    &\leqslant  ||T_{\varphi} k_z||_{\infty,\alpha} ||k_z||_{1,\alpha}\\
&\leqslant||T_{\varphi}k_z||_{\infty,\alpha}.
\end{align}
$T_{\varphi}$ is a compact operator, so $\lim\limits_{z\to \infty}||T_{\varphi}k_z||_{\infty,\alpha}=0$,\
then $ \tilde{\varphi}_{2}(z)\in C_0(\mathbb{C})$. By lemma 1.5, $\tilde{\varphi}_{t}(z)\in C_0(\mathbb{C})$ for any $t>0.$

On the contrary, we suppose $\{f_n\}$ is a bounded sequence in $F_{\alpha}^{\infty}$ and converges to $0$ uniformly on any compact subset of $\mathbb{C}$.
By Theorem 4.3
\begin{align}\nonumber
|T_{\varphi}f_n(z) e^{-\frac{\alpha|z|^2}{2}}|
&=|\int_{\mathbb{C} }\varphi(\omega) f_n(\omega)\overline{K_z(\omega)} \text{d}\lambda_\alpha(\omega) e^{-\frac{\alpha|z|^2}{2}}| \\
&\leqslant 2||f_n||_{\infty,\alpha} \tilde{\varphi}_{1}(z).
\end{align}
Because $\tilde{\varphi}_t\in C_0(\mathbb{C})$, then $ \tilde{\varphi}_{1}(z)\in C_0(\mathbb{C})$.
 Thus we know that $\lim\limits_{n\to \infty}||T_{\varphi}f_n||_{\infty,\alpha}=0$, then $T_{\varphi}$ is a compact operator on $F_{\alpha}^{\infty}$.\qedsymbol
\end{proof}
\begin{theorem}
 \textit{Suppose $\varphi \in BMO$. $ T_{\varphi}$ is a bounded operator on $F_{\alpha}^{\infty}$ if and only if $\tilde{\varphi}_{t}\in L^{\infty}(\mathbb{C},dA)$ for $t>0$..}
\end{theorem}
\begin{proof} Let $\varphi \in BMO$, then $\varphi$ satisfies the condition $(\uppercase\expandafter{\romannumeral1}_1)$ and
$\tilde{\varphi}_{t}\in L^{\infty}(\mathbb{C},dA)$. By theorem 3.34 in $[1]$,
$$||\varphi \circ \varphi_z-\tilde{\varphi}(z)||_{L^1(d\lambda_{\alpha})}\leqslant C,$$
where $\varphi_z(\omega)=z-\omega$.
Thus
    $$0\leqslant|| \varphi\circ \varphi_z ||_{L'(d\lambda_{\alpha})}-|\widetilde{\varphi}(z)|=\widetilde{ |\varphi|}_{2}(z)-|\widetilde{\varphi}_{2}(z)|\leq C.$$
Thus $\widetilde{\varphi}_{t}\in  L^{\infty}(\mathbb{C},\text{d}A )$ if and only if  $\widetilde{|\varphi|}_t\in  L^{\infty}(\mathbb{C},\text{d}A)$.
For any~$f\in F_{\alpha}^{\infty}$,
\begin{align}\nonumber
|T_{\varphi}f(z) e^{-\frac{\alpha|z|^2}{2}}|
&=|e^{-\frac{\alpha|z|^2}{2}}\int_{\mathbb{C} }\varphi(\omega) f(\omega)\overline{K_z(\omega)} \text{d}\lambda_{\alpha}(\omega)| \\
&=|\frac{\alpha}{\pi}\int_{\mathbb{C}} \varphi(\omega) f(\omega)\overline{K_z(\omega)}e^{-\alpha|\omega|^2 } e^{-\frac{\alpha|z|^2}{2}} \text{d}A(\omega) |\\
&\leqslant \frac{\alpha}{\pi} \int_{\mathbb{C}} |f(\omega)e^{-\frac{\alpha|\omega|^2}{2} }||\varphi(\omega)|
| \overline{K_z(\omega)} e^{-\frac{\alpha|\omega|^2}{2}-\frac{\alpha|z|^2}{2}}|\text{d}A(\omega)\\
&\leqslant \frac{\alpha}{\pi} ||f||_{\infty,\alpha} \int_{\mathbb{C}}|\varphi(\omega)|e^{-\frac{\alpha|z-\omega|^2}{2}}\text{d}A(\omega)\\
&\leqslant 2  ||f||_{\infty,\alpha} \widetilde{|\varphi|_1}(z).
\end{align}
Then~$||T_{\varphi}f(z)||_{\infty,\alpha}\leqslant 2|| ~\widetilde{|\varphi|_1}~||_{\infty}$, so ~$T_{\varphi}$~is a bounded operator on ~$F_{\alpha}^{\infty}$~.

On the contrary\ $T_{\varphi}$ is boundary, then
$$|\tilde{\varphi}_{2}(z)|=|<T_{\varphi}k_z,k_z>|\leqslant||T_{\varphi} k_z||_{\infty,\alpha}||k_z||_{1,\alpha}\leqslant||T_{\varphi}||.$$
This means
$\tilde{\varphi}_{2}\in L^{\infty}( \mathbb{C},\text{d}A)$, so $\tilde{\varphi}_{t}\in L^{\infty}( \mathbb{C},\text{d}A)$ for $t>0$.$\qedsymbol$
\end{proof}
\begin{theorem}
\textit{Let~$\varphi \in BMO$. $T_{\varphi}$ is a compact operator on $F_{\alpha}^{\infty}$, then ~$\tilde{\varphi}_{t}(z)\in C_0(\mathbb{C})$ for $t>0$.}
\end{theorem}
\begin{proof}Let~$f_n=k_z$. $T_{\varphi}$ is compact on~$F_{\alpha}^{\infty}$~, so~$\lim\limits_{z\to  \infty}||T_{\varphi}k_z||_{\infty, \alpha}=0$, and
\begin{align}\nonumber
|\tilde{\varphi}_{2}(z)|
&=|<T_{\varphi}, k_z>| \\
&\leqslant ||T_{\varphi}k_z||_{\infty,\alpha}||k_z||_{1,\alpha}\\
&\leqslant ||T_{\varphi}k_z||_{\infty,\alpha}.
\end{align}
We can get~$\tilde{\varphi}_{2}(z)\in C_0(\mathbb{C})$, by Lemma 1.5,~$\tilde{\varphi}_{t}(z)\in C_0(\mathbb{C})$ for $t>0$. \qedsymbol

\begin{theorem}
\textit{Let $\varphi \in BMO$. If ~$\widetilde{|\varphi|_{t}}(z)\in C_0(\mathbb{C})$, then $T_{\varphi}$ is a compact operator on $F_{\alpha}^{\infty}$.}
\end{theorem}
\textit{Proof} Let~$\{f_n\}_{n=1}^{+\infty}$ be a bounded  sequence in $F_{\alpha}^{\infty}$ and converge  to $0$ uniformly on any compact subset of $\mathbb{C}$~,
 thus
\begin{align}\nonumber
|T_{\varphi}f_{n}(z)e^{-\frac{\alpha|z|^2}{2}}|
&=|e^{-\frac{\alpha|z|^2}{2}}\int_{\mathbb{C}}\varphi(\omega)
f_n(\omega)\overline{K_z(\omega)}\text{d}\lambda_{\alpha}(\omega)| \\
&\leqslant  e^{-\frac{\alpha|z|^2}{2}}\int_{|\omega| \leqslant R}|\varphi(\omega)f_n(\omega)K_z(\omega)| \text{d}\lambda_{\alpha}(\omega)\\
 & +\frac{\alpha}{\pi}\int_{|\omega| > R}|f_n(\omega)e^{-\frac{\alpha|\omega|^2}{2}}| |\varphi(\omega)|e^{-\frac{\alpha|z-\omega|^2}{2}} \text{d}A(\omega).
\end{align}
Because $\{f_n\}_{n=1}^{+\infty}$~ is a bounded  sequence in $F_{\alpha}^{\infty}$ and converge  to $0$ uniformly on any compact subset of $\mathbb{C}$~,
 so for any $\varepsilon>0$, there is a natural number $N$, such that when $n>N$, $|f_n(\omega)|<\varepsilon$ for all $\omega \in B(0,R)=\{\omega||\omega|\leq R\}, R>0$.
 For $\varphi \in BMO$, then $\varphi$ satisfies the condition $(\uppercase\expandafter{\romannumeral1}_1)$, i.e., there is a constant $C>0$
 such that $\int_{\mathbb{C}}|\varphi(\omega)||K_z(\omega)|\text{d}\lambda_{\alpha}(\omega)\leq C.$
 Then
 \begin{align}\nonumber
 & e^{-\frac{\alpha|z|^2}{2}}\int_{|\omega| \leqslant R}|\varphi(\omega)f_n(\omega)K_z(\omega)| \text{d}\lambda_{\alpha}(\omega)\\
 &< \varepsilon\int_{\mathbb{C}}|\varphi(\omega)||K_z(\omega)|\text{d}\lambda_{\alpha}(\omega)\\
 &\leq \varepsilon\int_{|\omega| \leq R}|\varphi(\omega)||K_z(\omega)|\text{d}\lambda_{\alpha}(\omega)\\
  &\leq C\varepsilon.
 \end{align}
 And
  \begin{align}\nonumber
  &\frac{\alpha}{\pi}\int_{|\omega| > R}|f_n(\omega)e^{-\frac{\alpha|\omega|^2}{2}}| |\varphi(\omega)|e^{-\frac{\alpha|z-\omega|^2}{2}} \text{d}A(\omega)\\
  &\leq \frac{\alpha}{\pi}||f_n||_{\infty,\alpha}\int_{\mathbb{C}}|\varphi(\omega)|e^{-\frac{\alpha|z-\omega|^2}{2}} \text{d}A(\omega)\\
  &\leq 2||f_n||_{\infty,\alpha}\widetilde{|\varphi|_{1}}(z).
  \end{align}
  Thus
  $|T_{\varphi}f_{n}(z)e^{-\frac{\alpha|z|^2}{2}}\leq C\varepsilon+2||f_n||_{\infty,\alpha}\widetilde{|\varphi|_{1}}(z).$
If $\widetilde{|\varphi|_{1}}(z)\in C_0(\mathbb{C})$, so $\lim\limits_{n\to +\infty}||T_{\varphi}f_n||_{\infty, \alpha}=0$, then $T_{\varphi}$ is a compact operator on $F_{\alpha}^{\infty}$.\qedsymbol

\end{proof}

\end{document}